\documentclass{amsart}
\usepackage[margin=1.2in]{geometry}
\title{Iterative actions of normal operators}

\usepackage{hyperref}
\author[A. Aldroubi, C. Cabrelli, A.F.  \c{C}akmak, U. Molter, A. Petrosyan]{A. Aldroubi, C. Cabrelli,  A. F.  \c{C}akmak, U. Molter, A. Petrosyan}

\date{\vspace{-5ex}}
\usepackage{amsthm}
\usepackage{ dsfont }
\usepackage{amssymb}
\usepackage{enumerate}
\usepackage{graphicx}
\usepackage{float}
\usepackage{bbm}
\usepackage{amsmath}
\usepackage{comment}
\usepackage{hyperref}
\usepackage{listings}
\usepackage{color}
\usepackage[dvipsnames]{xcolor}

\hypersetup{
    colorlinks,
    citecolor=blue,
    filecolor=blue,
    linkcolor=blue,
    urlcolor=blue
}

\newtheorem{theorem}{Theorem}[section]
\newtheorem{proposition}[theorem]{Proposition}
\newtheorem{lemma}[theorem]{Lemma}

\newtheorem{corollary}[theorem]{Corollary}
\newtheorem{definition}[theorem]{Definition}
\newtheorem{remark}[theorem]{Remark}

\newtheorem*{prop}{Proposition}
\newtheorem{example}{Example}

\newcommand{\HH}{\mathcal{H}}

\newcommand{\cA}{\mathcal{A}}

\newcommand{\R}{\mathbb{R}}
\newcommand{\Z}{\mathbb{Z}}
\newcommand{\N}{\mathbb{N}}
\newcommand{\CC}{\mathbb{C}}

\providecommand{\norm}[1]{\lVert#1\rVert}
\newcommand{\G}{\mathcal{G}}
\newcommand {\D} {\mathbb D}

\newcommand{\wi}{\widetilde}
\newcommand{\Om}{\Omega}
\newcommand{\supp}{{\rm supp\,}}
\newcommand{\rank}{{\rm rank\,}}

\newcommand{\spn}{{\rm span\,}}

\newcommand {\la} {\langle}
\newcommand {\ra} {\rangle}
\begin{document}
	\date{}
\address{\textrm{(Akram Aldroubi)}
Department of Mathematics,
Vanderbilt University,
Nashville, Tennessee 37240-0001 USA}
\email{aldroubi@math.vanderbilt.edu}

\address{\textrm{(Carlos Cabrelli)}
Departamento de Matem\'atica,
Facultad de Ciencias Exac\-tas y Naturales,
Universidad de Buenos Aires, Ciudad Universitaria, Pabell\'on I,
1428 Buenos Aires, Argentina and
IMAS, UBA-CONICET, Consejo Nacional de Investigaciones
Cient\'ificas y T\'ecnicas, Argentina}
\email{cabrelli@dm.uba.ar}

\address{\textrm{(Ahmet Faruk   \c{C}akmak)}
Department of Mathematical Engineering, Yildiz Technical Univ., Davutpa\c{s}a Campus, Esenler, \.{I}stanbul, 80750, Turkey, Department of Mathematics,
Vanderbilt University,
Nashville, Tennessee 37240-0001 USA}
\email{acakmak@yildiz.edu.tr}

\address{\textrm{(Ursula Molter)}
Departamento de Matem\'atica,
Facultad de Ciencias Exac\-tas y Naturales,
Universidad de Buenos Aires, Ciudad Universitaria, Pabell\'on I,
1428 Buenos Aires, Argentina and
IMAS, UBA-CONICET, Consejo Nacional de Investigaciones
Cient\'ificas y T\'ecnicas, Argentina}
\email{umolter@dm.uba.ar}

\address{\textrm{(Armenak Petrosyan)}
Department of Mathematics,
Vanderbilt University,
Nashville, Tennessee 37240-0001 USA}
\email{armenak.petrosyan@vanderbilt.edu}

\thanks{The research of
A.~Aldroubi and A.~Petrosyan is supported in part by NSF Grant DMS- 1322099.
C.~Cabrelli and U.~Molter are partially supported by
Grants  PICT 2014-1480 (ANPCyT), CONICET PIP 11220110101018,
UBACyT 20020130100403BA and UBACyT 20020130100422BA. A.F.  \c{C}akmak is supported by the Scientific and Technological Research Council of Turkey, TUBITAK 2014 - 2219/I
}

\keywords{Sampling Theory, Frames, Sub-Sampling,
Reconstruction, M\"untz-Sz\'asz Theorem, Feichtinger conjecture}
\subjclass [2010] {46N99, 42C15,  94O20}

\begin{abstract}
Let $A$ be a normal operator in a Hilbert space $\HH$, and let $\G \subset \HH$ be a  countable  set of vectors. We investigate the relations  between $A$, $\G$  and $L$ that make the system of iterations $\{A^ng: g\in \G,\;0\leq n< L(g)\}$ complete, Bessel, a basis, or a frame for $\HH$. The problem is motivated by the dynamical sampling problem and is connected to several topics in functional analysis, including,  frame theory and spectral theory.  It  also has relations to topics in applied harmonic analysis including, wavelet theory and time-frequency analysis.
\end{abstract}
\maketitle

\section{Introduction}

Let $\HH$ be an infinite dimensional separable complex Hilbert space, $A\in B(\HH)$ a bounded normal operator and $\G$ a  countable (finite or countably infinite) collection of vectors in $\HH$. Let $L$ be a function $L:\G \rightarrow \N^*$, where $\N^* = \{1,2, \dots \}\cup \{+\infty\}$. We are interested in  the structure of the set of iterations of the operator $A$ when acting on the vectors in $\G$ and are limited by the function $L$. More precisely, we are interested in the following two questions:
\begin {enumerate}
\item [(I)]Under what conditions on $A,\, \G$ and $L$ is the iterated system of vectors
\begin{equation*}
\label {ITSYS}
\{A^ng: g\in \G,\;0\leq n< L(g)\}
\end {equation*}
 complete, Bessel, a basis, or a frame for $\HH$? 

\item [(II)]If $\{A^ng: g\in \G,\;0\leq n< L(g)\}$ is complete, Bessel, a basis, or a frame for $\HH$ for some system of vectors $\G$ and a function $L:\G \rightarrow \N^*$, what can be deduced about the operator $A$?
\end {enumerate}

We study these and other related questions and we give  answers in many important and general cases. In particular, we show that there is a direct relation  between the spectral properties of a normal operator and the properties of the systems of vectors generated by its iterative actions on a set of vectors. We are hoping that the questions above and the approach we use can be interesting for  research in both,  frame theory and operator theory.  

For the particular case when $L(g)=\infty$ for every $g \in \G$, we show that, if the system of iterations $\{A^ng: g\in \G,\;n \ge 0\}$ is complete and Bessel, then the spectral radius of $A$ must be less than or equal to $1$. Since $A$ is normal, $A$ must be a contraction in this case, i.e., $\|A\|\le1$. Moreover,  it's unitary part must be absolutely continuous (with respect to the Lebesgue measure on the circle). The converse of this is  also true:  for every normal contraction with absolutely continuous unitary part, there exists a set $\G$ such that  $\{A^ng: g\in \G,\;n \ge 0\}$ is a complete Bessel system.	If  $\{A^ng: g\in \G,\;n \ge 0\}$ is a frame then its unitary part must be 0 and the converse  is also true:  for every normal contraction with no unitary part there exists a set $\G$ such that  $\{A^ng: g\in \G,\;n \ge 0\}$ is a (Parseval) frame.

The questions above, in their formulation have similarities with  problems involving cyclical vectors in operator theory, and our analysis  relies on the spectral theorem for normal operators with multiplicity \cite {conway}. There have been some attempts to generalize multiplicity  theory to non-normal operators \cite{Nik89}. Although it cannot be generalized entirely,  some aspects of it have been extended to general operators. In finite dimensions, the spectral theorem for normal operators,   represents the underlying space as a sum of invariant subspaces. For general operators the decomposition into invariant subspaces leads to Jordan's theorem. In the infinite dimensional case,  the extension leads to a decomposition into invariant subspaces, and one of the goals is to give conditions under which these subspaces $\{S_n\}$ form  Riesz bases or equivalently unconditional bases, see \cite{Nik89, Tre97} and the references therein (this notion of Riesz basis is related but different from the one we use in this work  as defined by \eqref {RB}).  The multiplicity of a spectral value for a normal operators has also been extended. For general operators,  a global multiplicity (called multicyclicity) is  particularly useful in the context of control theory: using multicyclicity theory  for a completely non-unitary contraction $A$, a formula for $\min |\G|$ such that $\{A^ng: g\in \G,\;n \ge 0\}$ is complete in $\HH$ was obtained ( see  \cite{NV83, Nik86} and the references therein). For a normal operator $A$, this number can be deduced  from Theorem \ref {mainthm}.

  Our main goal in this paper is to find frames or other types of systems through the iterative action of a normal operator,  and we use the full power of the spectral theorem with multiplicity.
	 
For us, the motivation to study the iterative actions of normal operators  comes from sampling theory and related topics \cite { ABK08, CMO12, D09, LS15, NS10, BHP15, BHP15ii, GOR15, Sun10, Sun15, DJ15, DJ15II}. Specifically, the motivation derives from the so called  {\em dynamical sampling problem} \cite {ACMT14, ADK15,D14,GRUV15, HRLV10}: Let the initial state of a system be given by a vector $f$ in a Hilbert space $\HH$ and assume that the initial state is evolving under the action of a bounded operator $A\in \mathcal B (\HH)$ to the states 
$$f_n=Af_{n-1},\;f_0=f.$$
Given a finite or countably infinite set of vectors $\G$,  the problem is then to find conditions on $A$, $\G$ and $L:\G \rightarrow \N^*$,   that allow the recovery  of the function $f\in \HH$ from the set of samples
\begin{equation}\label{samples}
\{\la A^nf,g \ra: g\in \G,\; 0\leq n<L(g)  \}.
\end{equation}
 Let $X$ be the set $X=\{(g,k): g \in \G, \; 0\le k <L(g)\}$. Under  appropriate conditions on $A, L$ and $\G$, the sequence $\{\la A^nf,g \ra :(g,n) \in X\}$ belongs to $\ell^2(X)$. It is fundamental in applications that the reconstruction operator  $R: \ell^2(X) \to \HH$ given by $R(\la A^nf,g\ra )=f$ for all $f \in \HH$, exists and is well defined. Moroever, it is very important that it is bounded.  This is because very often, the samples $\{\la A^nf,g \ra :(g,n) \in X\}$ are corrupted by ``noise''  $\{\eta_{g,n} :(g,n) \in X\}$, and we require the reconstruction $\tilde f= R(\la A^nf,g\ra )+R (\eta_{g,n})$ to  be close to $f$ when the noise is small. When $R$ exists and is bounded, it is said that $f$ can be reconstructed in a stable way. If this property holds it is also said that the reconstruction is continuously dependent on the data.

It is not difficult to show that the problem above is related to the problem of finding  whether the set of vectors $\{(A^*)^n g\}_{ g\in \G,\;0\leq n< L(g)}$ (where $A^*$ denote the adjoint of $A$) is complete or a frame (see definition in Section \ref{sec2} below) for $\HH$ as described in the following proposition. 

\begin{prop}[\cite{ACMT14}]
For any vector  $f\in \HH$, $f$ can be uniquely recovered from the samples (\ref{samples}) if and only if the system of vectors $$\{(A^*)^n g\}_{ g\in \G,\;0\leq n< L(g)}$$
is complete in $\HH$. 

Moreover, $f$ can be uniquely recovered in a stable way from the samples (\ref{samples}) if and only if the system of vectors $$\{(A^*)^n g\}_{ g \in \G,\;0\leq n< L(g)}$$ is a frame for $\HH$.
\end{prop}

\subsection {Contribution and organization}
In this paper we consider both: Problem (I) and (II) above, in the general separable Hilbert space setting, and for general normal operators. Problem (I) has already been studied in [3] for the special case when $A\in \mathcal B(\HH)$ is a self-adjoint operator that can be unitarily mapped to an infinite diagonalizable matrix in $\ell^2(\N)$. Thus, all the results in  \cite {ACMT14} are subsumed by the corresponding theorems of this paper. The present paper contains new theorems that are not generalizations of those in  \cite {ACMT14}. In particular those related to Problem (II) and those that are connected to the action of a group of unitary operators.

%Problem (I) above  has been studied in \cite {ACMT14} for the special case when $A\in \mathcal B(\HH)$ is a self-adjoint operator that can be unitarily mapped to an infinite diagonalizable matrix in $\ell^2(\N)$. The results in this paper are for general separable Hilbert space and general normal operators. Thus, all the results in \cite {ACMT14} are subsumed by the corresponding theorems of this paper. However, this paper contains new theorems that are not generalizations of those in   \cite {ACMT14}. In particular those related to Problem (II) described above and those that are connected to the action of a unitary group of operators.

In Section \ref {completeness},  we characterize all countable subsets $\G\subset \HH$ such  that $\left\{A^ng\right\}_{g\in\G,\;0\leq n< \infty}$ is complete in $\HH$ when the operator $A$ is a normal reductive operator (Theorem \ref {mainthm}). These results are also extended to the system of vectors $\left\{A^ng\right\}_{g\in\G,\;0\leq n< L(g)}$, where $L$ is any suitable function from $\G$ to $\N^*$.  However, we also show that the system $\left\{A^ng\right\}_{g\in\G,\;0\leq n< \infty}$ always fails to be a basis for $\HH$ when $A$ is a normal operator (Corollary \ref {NRB}). In fact, if the set $\G\subset \HH$ is finite, and $A$ is a reductive normal operator, then $\left\{A^ng\right\}_{g\in\G,\;0\leq n< L(g)}$ cannot be a basis for $\HH$ for any choice of the function $L$ (Corollary \ref {NotRB}). The  obstruction to being a basis is the redundancy in the form of non-minimality of the set of vectors $\left\{A^ng\right\}_{g\in\G,\;0\leq n< L(g)}$. Thus, the set $\left\{A^ng\right\}_{g\in\G,\;0\leq n< \infty}$, (or $\left\{A^ng\right\}_{g\in\G,\;0\leq n< L(g)}$) may be a frame, but cannot be a basis. It turns out that, in general, it is  difficult for a system of vectors of the form $\left\{A^ng\right\}_{g\in\G,\;0\leq n< \infty}$ to be a frame as is shown in Section \ref {BS}. The difficulty is that the spectrum of $A$ must be very special as can be seen from Theorem \ref {framecond}. Such frames however do exist, as shown by the constructions in \cite {ACMT14}.  Surprisingly, the difficulty becomes an obstruction if we normalize the system of iterations to become $\left\{\frac{A^n g}{\norm{A^n g}}\right\}_{g\in \G,\;n\geq 0}$ when the operator $A$ is self-adjoint as described in Section \ref {SA} (Theorem \ref {normframe}).  In Section \ref {GUO} we apply our results to systems that are generated  by the unitary actions of a discrete group $\Gamma$ on a set of vectors $\G\subset \HH$ which is common in many constructions of wavelets and frames.

\section {Notations and preliminaries}\label{sec2}

 Let $I$ be a countable indexing set, and $\HH$ a separable Hilbert space. Recall that a system of vectors  $\{f_i\}_{i\in I} \subset \HH$ is a Riesz sequence  in $\HH$ if  there exist two constants $m, M>0$ such that
\begin {equation} 
\label {RB}
m\|c\|^2_{\ell^2(I)}\le \|\sum_{i\in I} c_if_i\|^2_\HH\le M\|c\|^2_{\ell^2(I)} \quad \text { for all } c \in \ell^2(I).\end {equation} 
If, in addition,   $\{f_i\}_{i\in I}  \subset \HH$ is complete then it is called a Riesz basis for $\HH$. 

 A sequence $\{f_i\}_{i\in I}  \subset \HH$ is said to be a Bessel system in $\HH$ if there exists a constant $\beta \ge 0$ such that

$$\sum_{i\in I} |\la f, f_i\ra |^2\le \beta \|f\|^2_\HH \quad \text { for all } f \in \HH,$$
and it is said to be  a frame for $\HH$ if there exist two constants $\alpha, \beta>0$ such that

$$ \alpha \|f\|^2_\HH \le \sum_{i\in I} |\la f, f_i\ra|^2\le \beta \|f\|^2_\HH \quad \text { for all } f \in \HH.$$
If $\alpha=\beta$ we say that the system is a Parseval frame.
The notion of frames  introduced by Duffin and Schaeffer \cite {DS52}, generalizes the notion of Riesz bases. In particular, if  $\{f_i\}_{i\in I}  \subset \HH$ is a frame for $\HH$, then any vector $f\in \HH$ has the representation $f=\sum_{i\in I}\la f,f_i\ra \tilde f_i$ where $\{\tilde f_i \}_{i\in I} \subset \HH$ is a dual frame. It is well-known that a Riesz basis is a frame but the converse is not necessarily true because  the vectors in a frame set $\{f_i \}_{i\in I} \subset \HH$ may have linear dependencies.  Further properties of frames, bases, and Bessel sequences can be found in \cite {DL00,C08,Heil11}.

One of the main tools that we use is the spectral theorem with multiplicity for normal operators described below. 

Let $\mu$ be a non-negative regular Borel measure on $\mathbb{C}$ with compact support $K$. Denote by $N_{\mu}$ the operator
$$N_{\mu}f(z)=zf(z),\;z\in K$$
acting on functions $ f \in L^2(\mu)$ (i.e. $f:\CC \rightarrow \CC$ , measurable with  $\int_\CC |f(z)|^2 d\mu(z) < \infty.$)

For a Borel non-negative measure $\mu$, we will denote by $[\mu]$ the class of Borel measures that are mutually absolutely continuous with  $\mu$.

\begin{theorem} [Spectral theorem with multiplicity]\label{decomp} For any normal operator $A$ on $\HH$ there are mutually singular compactly supported non-negative Borel measures  $\mu_j,\;1\leq j\leq\infty$, such that $A$ is  equivalent to the operator
$$ N^{(\infty)}_{\mu_\infty}\oplus N_{\mu_1}\oplus N^{(2)}_{\mu_2}\oplus\cdots$$
	i.e. there exists a unitary transformation  
$$U:\HH\rightarrow(L^2(\mu_{\infty}))^{(\infty)}\oplus L^2(\mu_1)\oplus (L^2(\mu_2))^{(2)}\oplus\cdots$$ 
such that	
	\begin{equation} \label{repres}
	UAU^{-1}=N^{(\infty)}_{\mu_{\infty}}\oplus N_{\mu_1}\oplus N^{(2)}_{\mu_2}\oplus\cdots.
	\end{equation}	
Moreover, if $M$ is another normal operator with corresponding measures $\nu_\infty,\nu_1,\nu_2,\dots$ then $M$ is  unitarily equivalent to $A$ if and only if  $[\nu_j]=[\mu_j],\;j=1,\dots,\infty$.
\end{theorem}

A proof of the theorem can be found in \cite{conway} (Ch. IX, Theorem 10.16) and \cite{conway1} (Theorem 9.14). 

Since the measures $\mu_j$ are mutually singular, there are mutually disjoint Borel sets $\{\mathcal{E}_j\}$ such that $\mu_j$ is concentrated on $\mathcal{E}_j$ for every $1\leq j\leq \infty$. 

We will define the scalar measure $\mu$, (usually called {\it the scalar spectral measure) }associated with the normal operator $A$ to be
\begin{equation} \label{themeas}
\mu:=\sum_{1\leq j\leq\infty} \mu_j.
\end{equation}

The Borel function function $J:\CC\rightarrow \N^*\cup\{0\}$ given by
\begin{equation}\label{multfunc}
J(z)=\begin{cases} 
j, & z\in \mathcal{E}_j \\
0, & otherwise 
\end{cases}
\end{equation}
		is called multiplicity function of the operator $A$. 
%	Thus every normal  operator is uniquely, up to a unitary equivalence, determined by a pair $(n, [\mu])$ where $n:\CC\rightarrow \{1,2,\dots,\infty\}$ is a $\mu$ measurable function and $[\mu]$ is the class of  measures 

  From Theorem \ref{decomp}, every normal  operator is uniquely determined, up to a unitary equivalence,  by the pair $([\mu],J)$. 
	
For $j\in \N$, define  $\Om_j$ to be the set $ \{1,...,j\}$  and $\Om_\infty$ to be the set $\N$. Note that $\ell^2(\Om_j) \cong \CC^j$, for $j\in \N$, and $\ell^2(\Omega_\infty) = \ell^2(\N).$ For $j=0$ we define $\ell^2(\Omega_0)$ to be the trivial space $\{0\}$.
	
%We denote by $\N^* := \N\cup \infty$. 
Let $\mathcal{W}$ be the Hilbert space 
$$\mathcal{W} :=(L^2(\mu_\infty))^{(\infty)}\oplus L^2(\mu_1)\oplus (L^2(\mu_2))^{(2)}\oplus\cdots$$ 
associated to the operator $A$ and
let $U:\HH \rightarrow \mathcal{W}$ be the unitary operator given by Theorem~\ref{decomp}.  If $g\in \HH$, we will denote by $\wi{g}$
the image of $g$ under $U$. Since $\wi{g} \in \mathcal{W}$ we have  $\wi{g} = (\wi{g}_j)_{j\in \N^*}$, where 
$\wi{g}_j$ is the restriction of $\wi g$ to $(L^2(\mu_j))^{(j)}$. Thus, for any $j\in \N^*$, $ \wi{g}_j $ is a function from $ \CC $ to $ \ell^2(\Omega_j) $ and
$$\sum_{j\in \N^*} \quad \int_\CC \|\wi{g}_j(z)\|_{\ell^2(\Om_j)}^2d\mu_j(z) <\infty .$$
Let $P_j$ be the projection defined for every  $\wi g \in \mathcal W$  by $\wi f=P_j\wi g$ where $\wi f_j=\wi g_j$ and $\wi f_k=0$ for $k\ne j$. 

%that send $\wi g$ to $P_j\wi g$ where $P_j\wi g $$L =P_j\wi g$ with $P_j$ being the projection of $ \mathcal{W} $ to  $(L^2(\mu_j))^{(j)}$. 

   Let $E$ be the spectral measure for the normal operator $A$. Then for every $\mu$-measurable set $G\subseteq \CC$ and vectors $f,g$ in $\HH$ we have the following formula
$$\la E(G)f,g\ra _{\HH}\; = \int_G\left[\sum_{1\leq j\leq\infty} \mathbbm{1}_{\mathcal{E}_j}(z)\la \wi f_j(z),\wi g_j(z)\ra _{\ell^2(\Om_j)}\right]d\mu(z),$$
which relates the spectral measure of $A$ with the scalar spectral measure of $A$.

In \cite{kriete} and \cite{krieteabr} the spectral multiplicity of multiplication operator is computed.

 As a generalization of self-adjoint operators, we will consider normal reductive operators. Reductive operators were first studied by P. Halmos \cite{halmos}  and J. Wermer \cite{wermer}.
 
 \begin{definition}
 	A closed subspace $\mathcal{V}\subseteq \HH$ is called reducing for the operator $A$ if both $\mathcal{V}$ and its orthogonal complement $\mathcal{V}^\perp$ are invariant subspaces of $A$.
 \end{definition}
 
 Notice that, $\mathcal{V}\subseteq \HH$ being reducing subspace for $A$ is equivalent to $\mathcal{V}$ being an invariant subspace both for $A$ and its adjoint $A^*$ and also equivalent to $AP_{\mathcal{V}}=P_{\mathcal{V}}A$ where $ P_{\mathcal{V}} $ is the projection operator onto $ \mathcal{V} $.
 
\begin{definition}
An operator $A$ is called reductive if every invariant subspace of $A$ is reducing.
\end{definition}

It is not known whether every reductive operator is normal. In fact, every reductive operator being normal is equivalent to the veracity of the long standing  {\em invariant subspace conjecture}, which states that every bounded operator on a separable Hilbert space has a non-trivial closed invariant subspace \cite{subspconj}.

\begin{proposition} \label{prop1}\cite {Kub} A normal operator is reductive if and only if its restriction to every  invariant subspace is normal.	
\end{proposition}

\begin{proposition}[\cite{wermer}] \label{prop2} Let $A$  be a normal operator on the Hilbert space $\HH$ and let $\mu_j$ be the measures in the representation \eqref{repres}  of $A$. Let $\mu$ be as in \eqref {themeas}.
Then $A$  is reductive if and only if for any two vectors $f,g\in \HH$ 
	$$\int_\CC z^n\left[\sum_{1\leq j\leq\infty} \mathbbm{1}_{\mathcal{E}_j}(z)\la \wi g_j(z),\wi f_j(z)\ra _{\ell^2(\Om_j)}\right]d\mu(z)=0$$
	for every $n\geq 0$ implies $\mu_j$-a.e. $\la \wi g_j(z),\wi f_j(z)\ra_{\ell^2(\Om_j)}=0$ for ever $j\in \N^*$.
\end{proposition}

Note that the property in the above proposition is equivalent to the implication 
	$$\int_\CC z^nh(z)d\mu(z)=0\; \text{and}\;h\in L^1(\mu)\Rightarrow  h=0\;\mu-\text{a.e.}$$
%	In fact, being reductive is equivalent to polynomials being weak-* dense in $L^\infty(\mu)$ according to \cite{sarason} p. 14.
%	%\textcolor{blue}{what we wrote isn't exactly saying this? so I guess no need to put a reference then just say "i.e. being reductive is equivalent..." }. 
%	This, in particular, means that the reductivity is a property of $[\mu]$ only and it is independent of the multiplicity function $J(z)$.

As proved in \cite{wermer}, being reductive is not entirely a property of the spectrum: it is possible to find  two operators with the same spectrum such that one is reductive the other is not. However the following sufficient condition holds

	\begin{proposition}[\cite{wermer}] \label{emptintspec}
		Let $A$ be a normal operator on $\HH$ whose spectrum $ \sigma (A)$ has empty interior and $\CC - \sigma(A)$ is connected. Then $A$ is reductive.
	\end{proposition}
	
	\begin{corollary}
		Every self-adjoint operator on a Hilbert space is reductive.
	\end{corollary}
	The fact that self-adjoint operators are reductive is easily derived without the use of Proposition \ref {emptintspec}.  However, to see how this fact follows form Proposition \ref {emptintspec}, simply note that for a self-adjoint operator $A$, $\sigma (A)$ is a compact subset of $\R$ hence, it has empty interior (as a subset of $\CC$), and $\CC - \sigma(A)$ is connected.
	
	Also the following necessary condition for being reductive holds. 
	\begin{prop}[\cite{scroggs}]
Let $A$ be  a normal operator. If the interior of $\sigma(A)$ is not empty then $A$ is not reductive.
	\end{prop}

\section{Complete systems with iterations }
\label {completeness}

This section is devoted to  the characterization of completeness of the system $\left\{A^ng\right\}_{g \in \G,\;n\ge 0}$ where $A$ is a reductive normal operator on a Hilbert space $\HH$ and $\G$ is a set of vectors in $\HH$.  This is done by ``diagonalizing'' the operator $A$ using multiplicity theory for normal operators, and the properties of reductive operators.

\begin{theorem} \label{mainthm} Let $A$ be a  normal operator on a Hilbert space $\HH$,  and let $\G$ be a countable set of vectors in $\HH$ such that  $\left\{A^ng\right\}_{ g \in \G,\;n\ge 0}$ is complete in $\HH$. Let  $\mu_\infty,\mu_1,\mu_2,\dots$ be the measures in  the  representation \eqref{repres}  of the operator $A$. Then for every  $1\leq j\leq\infty$ and $\mu_j$-a.e. $z$, the system of vectors  $\{\wi g_j(z)\}_{g\in \G}$ is complete in   $\ell^2\{\Om_j\}$.
%(where $\ell^2\{1,2,\dots,j\}\cong \mathbb{C}^j$ and $\ell^2\{1,2,\dots,\infty\}\cong \ell^2(\mathbb{Z}_+)).$

 If in addition to being normal,  $A$ is also reductive,  then $\left\{A^ng\right\}_{ g\in \G,\;n\ge 0}$		being complete in $\HH$ is equivalent to $\{\wi g_j(z)\}_{g\in \G}$ being complete in   $\ell^2\{\Om_j\}$  $\mu_j$-a.e. $z$  for every  $1\leq j\leq\infty$.	
\end{theorem}

\begin{example}\label{examp1}
Let $A$ be a convolution operator on $\HH=L^2(\R)$ given by $Af=a\ast f$, where $a\in L^1(\R)$ is a real valued, even function (hence  the Fourier transform $\hat a$ of $a$ is real valued even function), such that $\hat a$ is strictly decreasing on $[0,\infty)$. For example, $A$ can be the discrete-time heat evolution operator given by the convolution with the Gaussian kernel  $a(x)=\frac 1 {\sqrt {4 \pi }}e^{-\frac {x^2} {4}}$. Since  $a\in L^1(\R)$,  $\hat a$ is continuous, and the spectrum of $A$ is the compact interval $I=[0,\frac 1 {\sqrt {4 \pi }}] \subset \R$. Hence as a subset of $\CC$, $I$ satisfies the assumption of Proposition \ref {emptintspec} and thus $A$ is reductive. Moreover, the facts that  $\hat a$ is real valued, even function, strictly decreasing on $[0,\infty)$, imply that $\mu_j=0$ for $j\ne 1,2$.   In fact, using \cite[Theorem 5]{krieteabr}, we get that $\mu_j=0$ for $j\ne 2$.
Then, using Theorem \ref{mainthm}, for a set of functions $\G\subset L^2(\R)$, the system of iterations $\left\{A^ng\right\}_{ g \in \G,\;n\ge 0}$ is complete in $L^2(\R)$ if and only if $\left\{\left(\hat g(\xi),\hat g(-\xi)\right)\right\}_{g\in\G}$ is complete in $\R^2$ for a.e. $ x\in \R $.
\end{example}
\begin {definition}  \label {ClassL}
For a given set $\G$, let  $\mathcal{L}$ be the class of functions $L:\G \rightarrow \N^*$ such that 
\begin{equation} \label{Lclass}
cl({\spn}\{A^ng\}_{g\in \G,\;0\leq n< L(g)})=cl({\spn}\{A^ng\}_{g\in \G,\;n\ge 0}).
\end{equation}
\end {definition}
\begin {remark}\rm \label{RemLclass}
Note that condition \eqref{Lclass} is equivalent to 
\begin{equation} \label{Lclass2}
A^{L(h)}h\in cl({\spn}\{A^ng\}_{g\in \G,\;0\leq n< L(g)}).
\end{equation}
for every $h\in \G$ such that $L(h)<\infty.$

In particular, $\mathcal{L}$ contains the constant function $L(g)=\infty$ for every $g \in \G$.   It also contains the function
\begin{equation} \label{lg}
l(g)=\min\left\{\left\{m \;|\; A^mg\in span\{g,Ag,\dots,A^{m-1}g\}\right\}, \infty\right\} \quad \text { for every } g \in \G.
\end{equation}
When $ l(g) $ is finite, it is called {\em the degree of the annihilator} of $g$.
 
 Because of condition \eqref{Lclass}, the reduced system $\{A^ng\}_{g\in \G,\;0\leq n< L(g)}$ will be complete in $\HH$ if and only if  $\{A^ng: g\in \G,n\geq 0\}$  is complete in $ \HH $.  Therefore, Theorem \ref{mainthm}  holds if we replace $\{A^ng\}_{g\in \G,n\geq 0}$  by  $\{A^ng\}_{g\in \G,\;0\leq n< L(g)}$ as long as $L\in \mathcal L$. 
 
Although when $L \in \mathcal L$,  $\{A^ng\}_{g\in \G,\;n\ge 0}$ and  $\{A^ng\}_{g\in \G,\;0\leq n< L(g)}$ are either both complete or both incomplete,  the system $\{A^ng\}_{g\in \G,\;0\leq n< L(g)}$ may form a frame while $\{A^ng\}_{g\in \G,\;n\ge 0}$ may not, since the possible extra vectors in  $\{A^ng\}_{g\in \G,\;n\ge 0}$ may inhibit the upper frame bound.  This difference in behavior between the two systems makes it important to study $\{A^ng\}_{g\in \G,\;0\leq n< L(g)}$ for $L\in \mathcal L$.
\end{remark}

\begin{example}
Let $\HH=\ell^2(\Z)$ and  $A$ be convolution operator with a kernel $a\in \ell^1(\Z)$, i.e. $ Af=a*f $. Let $\G=\{e_{mk}\}_{k\in \Z}$ for some $m>1$  where $ \{e_k\}_{k\in\Z} $ is the canonical basis of $\ell^2(\Z)$.
The  Fourier transform  of $a$ is defined as
$$\hat a(\xi)=\sum\limits_{k\in \mathbb{Z}} a(k) e^{-2\pi i\xi k},\; \xi \in [0,1].$$
Denote
$$\mathcal A_m(\xi)= \left(
\begin{matrix} 1 & 1 & \hdots &1\\
\hat{a}(\frac{\xi}{m}) & \hat{a}(\frac{\xi+1}{m}) & \hdots & \hat{a}(\frac{\xi+m-1}{m})\\
\vdots & \vdots& \vdots & \vdots \\
\hat{a}^{(L-1)}(\frac{\xi}{m}) & \hat{a}^{(L-1)}({\frac{\xi+1}{m}}) & \hdots & \hat{a}^{(L-1)}(\frac{\xi+m-1}{m})\end{matrix}\right).
$$  
Let $\sigma(\xi)$ denote the smallest singular value of the matrix $ \mathcal A_m(\xi) $.
Let $L(g)=M$ for each $g \in \G$. From \cite{ADK15}, the system $\{A^ng\}_{g \in \G, 0\leq n< M}$ is complete in   $\ell^2(\Z)$ if and only if $ \mathcal A_m(\xi) $ has a left inverse for a.e. $\xi \in [0,1]$, or equivalently $\sigma(\xi)>0$ for a.e. $\xi \in [0,1]$, and it forms a frame if and only if  $\sigma(\xi)\ge\alpha$ for a.e. $\xi \in [0,1]$ for some $\alpha > 0.$ Since $ \mathcal A_m(\xi) $ is a Vandermonde matrix, iterations $n> m-1$ will not affect the completeness of the system. Thus,  we let   $ M=m$.
In that case $\{A^ng\}_{g \in \G, 0\leq n\leq m-1}$ is complete in   $\ell^2(\Z)$ if and only if  $\det \cA_{m}(\xi)\neq 0$ for a.e $\xi \in [0,1]$, and it is a frame if and only if for a.e $\xi \in [0,1]$,   $ |\det \cA_{m}(\xi)| \ge \alpha \}$ for some $\alpha > 0.$
 
Although there are infinitely many convolution operators that satisfy this last condition, many natural operators in practice do not. For example, an operator where  $a$ is real, even and $\hat a$ is strictly decreasing on $[0,\frac12]$. For this case, it can be shown that the matrices $\mathcal A_m(0)$ and $\mathcal A_m(\frac 12)$ are singular,  while all the other matrices   $\mathcal A_m(\xi)$ are invertible. For this case any set of the form $\G=\{e_{mk}\}_{k\in \Z}\cup \{e_{mlk+1}\}_{k\in \Z}$ where $l\ge1$, produces a system $\{A^ng\}_{g \in \G, 0\leq n\leq m-1}$ which is a frame for $\HH=\ell^2(\Z)$. 
%To recover the function $f$ in a stable way, we need to make additional samples that are well placed. Let $T_c$ be the operator that shifts a vector in $\ell^2(\Z)$ to the right by $c$ units so that, for  $f\in l^2(\mathbb{Z})$,  $T_cf(k)=f(k-c)$. The following theorem was proved in \cite{ADK15}.
%\begin {theorem} \label{addsampZ}
%Suppose $\hat a$ is real, symmetric in $[0,1]$ around $\frac{1}{2}$, continuous, and strictly decreasing on $\left[0,\frac{1}{2}\right]$, $m$ and $J$ are odd, and $\Omega = \{1, \hdots, \frac{m-1}{2}\}$. Then the additional sampling given by $\left\{S_{mJ}T_c\right\}_{c\in\Omega}$ combined with samples \eqref{onedimZ} provide stable recovery for $f$.
%\end{theorem}

\end{example}

The proof of Theorem \ref {mainthm} below, also shows that, for normal reductive operators, completeness in $\HH$ is equivalent to  the system  $\left\{N^n_{\mu_j}\wi{g}_j\right\}_{g\in \G,\;n\ge 0}$ being complete in $(L^2(\mu_j))^{(j)}$ for every $1\leq j\leq\infty$, i.e. the completeness of $\left\{A^ng\right\}_{g \in \G,\;n\ge 0}$ is equivalent to the completeness of its projections  onto the mutually orthogonal subspaces $UP_jU^*\HH$  of $\HH$. This should be contrasted to the fact that, in general, completeness of a set of vectors $\{h_n\} \subset \HH$ is not equivalent to the completeness of its projections on subspaces whose orthogonal sum is $\HH$. We have

\begin {theorem} Let $A$ be a  normal reductive operator on a Hilbert space $\HH$,  and let $\G$ be a countable system of vectors in $\HH$. Then, 
$\left\{A^ng\right\}_{g \in \G,\;n\ge 0}$ is  complete in $\HH$ if and only if  the system $\left\{N^n_{\mu_j}\wi{g}_j\right\}_{g \in \G,\;n\ge 0}$ is complete in $(L^2(\mu_j))^{(j)}$ for every $1\leq j\leq\infty$.
\end{theorem}

\begin{proof} [Proof of Theorem \ref {mainthm}]   Since $\left\{A^ng\right\}_{g \in \G,\;n\ge 0}$ is complete in 
$\HH$,\[U\{A^ng: g\in \G,n\geq 0\}=\{(N^n_{\mu_j}\wi{g}_j)_{ j \in \N^*}: g \in \G , n\geq 0\}\] is complete in $\mathcal W=U\HH$. 
Hence, for every $1\leq j\leq\infty$, the system  $\mathcal S=\left\{N_{\mu_j}^n\wi{g}_j\right\}_{ g \in \G,\;n\ge 0}$ is complete in  $(L^2(\mu_j))^{(j)}$.
	
To finish the proof of the first statement of the theorem we use the following lemma which is an adaptation of  \cite [Lemma 1]{kriete}.
	
	\begin{lemma} \label{auxlem} Let $\mathcal S$ be  a complete countable set of vectors in $(L^2(\mu_j))^{(j)}$, then for $\mu_j$-almost every $z$ $\{h(z): h\in \mathcal S\}$  is complete in  $\ell^2(\Om_j)$.
	\end{lemma}
 Since  $S$	is complete in $(L^2(\mu_j))^{(j)}$, Lemma \ref {auxlem} implies that $\{z^n\wi g_j(z)\}_{ g \in \G,\;n\ge 0}$ is complete in  $\ell^2(\Om_j)$ for each $j \in \N^*$. 
But $span \{z^n\wi g_j(z)\}_{ g \in \G,\;n\ge 0}=span \{\wi g_j(z)\}_{ g \in \G\;}$. Thus,  we have proved the first part of the theorem.

	Now additionally assume that $A$ is also reductive. Let $$\wi f\in (L^2(\mu_\infty))^{(\infty)}\oplus L^2(\mu_1)\oplus (L^2(\mu_2))^{(2)}\oplus\cdots$$
	and 
	$$\la UA^ng,\wi f \ra =\sum_{1\leq j\leq\infty}\int_\CC z^n\la \wi g_j(z),\wi f_j(z)\ra _{\ell^2(\Om_j)}d\mu_j(z)=0$$
	for every $g \in \G$ and every $0\leq n <\infty$. Since the measures $\mu_j, \;1\leq j\leq\infty,$ are mutually singular, we get that
	\begin{eqnarray} \label{zeroness}
		\lefteqn{\sum_{1\leq j\leq\infty}\int_\CC z^n<\la \wi g_j(z),\wi f_j(z)\ra_{\ell^2(\Om_j)}d\mu_j(z)} \\
	&=\int_\CC z^n\left[\sum_{1\leq j\leq\infty} \mathbbm{1}_{\mathcal{E}_j}\la \wi g_j(z),\wi f_j(z)\ra _{\ell^2(\Om_j)}\right]d\mu(z) \notag
	\end{eqnarray}
	 for every $g \in \G$ and every $n\geq 0 $ with $\mu$ as in \eqref{themeas}.

Using the fact that the operator $A$ is reductive, from Proposition \ref{prop2}, we conclude that 
	$$ 
	\la \wi g_j(z),\wi f_j(z)\ra_{\ell^2(\Om_j)}=0 \;\; \mu_j\text{-a.e.}\; z .
	$$
	Since, by assumption $\{\wi g_j(z)\}_{g\in \G}$ is complete in  $\ell^2(\Om_j)$ for $\mu_j$-a.e $z$, we obtain 
	$$\wi f_j=0 \;\;\mu_j\text{-a.e.}\; z \quad \text{	for every $ j\in \N^*$.} $$
Thus $\wi f=0$ $\mu$-a.e.,  and therefore,  $\left\{A^ng\right\}_{g \in \G,\;n\ge 0}$		is complete in $\HH$.
\end{proof}

\section{Minimality property and basis}

The goal of this  section is to study the conditions on the operator $A$ and the set of vectors $\G$ such that the system $\left\{A^ng\right\}_{g\in \G,\;0\leq n< L(g)}$ is minimal or a basis  for $\HH$. We start with the following proposition:
\begin{proposition} \label{nonmini}
	If  $A$ is a  normal operator on $\HH$ then, for any set of vectors  $\G \subset \HH$,    the system of iterates $\left\{A^ng\right\}_{g\in\G, n\ge 0}$ is not a complete and minimal system in $ \HH $.
\end{proposition}
Note that Proposition \ref{nonmini} is trivial if the $\dim \HH<\infty$ and becomes interesting only when $\dim \HH=\infty$. As a corollary of Proposition~\ref{nonmini} we get
\begin {corollary} \label {NRB}
If  $A$ is a  normal operator on $\HH$ then, for any set of vectors  $\G \subset \HH$,    the system of iterates $\left\{A^ng\right\}_{g\in\G, n\ge 0}$ is not a basis for $ \HH $.
\end {corollary}
 If we remove the completeness condition in the statement of Proposition \ref {nonmini} above,  then the operator $Af=zf$ on the unit circle with arc length measure gives an orthogonal system  when iterated on the vector $g\equiv 1$, i.e., for this case $\left\{z^ng \right\}_{n\ge 0}$ is minimal since it is an orthonormal system. However, if in addition to being normal, we assume that $A$ is reductive then the statement of proposition \ref {nonmini} remains true  without the completeness condition since, by Proposition \ref{prop1},  the restriction of $A$ onto   $cl({\spn}\{{A^ng}_{g\in \G, n\geq 0})$ will be a normal operator and we will have a minimal complete system  contradicting the claim of Proposition \ref {nonmini}. Thus, we have the following corollary
 
 \begin{corollary}\label{reductmin}
	If  $A$ is a  reductive normal operator on $\HH$,  then, for any countable system of vectors $\G \subset \HH$,     the system of iterates $\left\{A^ng\right\}_{g\in\G, n\ge 0}$ is  not a minimal system.
 \end{corollary} 

As another corollary of  Proposition \ref{nonmini}, we get

\begin{corollary} \label{nonmn}
	Let $A$ be a  reductive normal operator on $\HH$,  $\G$ a countable system of vectors in $\HH$ and let $L\in \mathcal{L}$. If for some $h\in \G$, $L(h)=\infty$,  then the system $\left\{A^ng\right\}_{g\in \G,\;0\leq n< L(g)}$ is not a  basis for $\HH$.
\end{corollary}
\begin{proof}
Let $V=cl({\spn}\{A^nh\}_{n\geq 0})$ where $L(h)=\infty$. $V$ is a closed invariant subspace for $A$ hence, by Proposition \ref{prop1}, the restriction of $A$ on $V$ is also normal, therefore, from Proposition \ref{nonmini}, $\{A^nh\}_{n\geq 0} $ is not minimal.
\end{proof}

In particular, since  $\dim \HH=\infty$ (the assumption in this paper),  if $|\G|<\infty$, then there exists $g\in \G$ such that $L(g)=\infty$. Thus we have 

\begin{corollary} \label {NotRB}
Let $ A $ be a reductive normal operator. If $|\G|<\infty$, then for any $L\in \mathcal{L}$ the system $\left\{A^ng\right\}_{g\in \G,\;0\leq n< L(g)}$ is never  a  basis for $\HH$.
\end{corollary} 
\begin{proof}[Proof of Proposition \ref{nonmini}] We prove that	if $\{A^ng\}_{g\in \G,n\geq 0}$ is complete in $\HH$, then for any $m\geq 0$, $ \{A^ng\}_{g\in \G,n=0,m,m+1,\dots} $ is also complete in $ \HH $, which implies non-minimality.

Assume $\{A^ng\}_{g\in \G,n\geq 0}$ is complete in $\HH$. Let $ \delta>0 $ and $f\in \HH$ be a vector such that  $\wi f(z)=0$ for any $z\in \overline {\mathbb D}_\delta$ where $ \overline {\mathbb D}_\delta$ is the closed unit disc of radius $\delta$ centered at $ 0 $. Then for a fixed $ m $, $\frac{\wi f}{z^m}$ is in $U\HH$ and hence can be approximated arbitrarily closely  by finite linear combinations of the vectors in $\{z^n\wi g\}_{g\in \G,n\geq 0}$.
Let $\wi f^{(1)} ,  \wi f^{(2)},\dots$ be a sequence in $U\HH$ such that $\wi f^{(s)} \rightarrow \frac{\wi f}{z^m}$ in $ U\HH $ and $\wi f^{(s)}$ is a finite linear combinations of the vectors in  $\{z^n\wi g\}_{g\in \G,n\geq 0}$
 for each $s$. Since $z^m$ is bounded on the spectrum of $A$,  it follows that $z^m \wi f^{(s)}\rightarrow \wi f$. Finally, we  note that $z^m\wi f^{(s)}$ is a finite linear  combination of the vectors  $ \{z^n\wi g\}_{g\in \G,n\geq m} $.
 	
 For a general $f \in \HH$, we have that 
\begin{equation}\label{delta-aprox}\lim_{\delta\rightarrow 0} \|\wi f-\wi f \mathbbm{1}_{\overline {\mathbb D}_{\delta}^c}\|^2_{L^2(\mu)}=\sum_{j\in\N^*} \|\wi f_j(0)\|_{\ell^2(\Omega_j)}^2\mu(\{0\})=  \mu_{J(0)}(\{0\})\|\wi f_{J(0)}(0)\|_{\ell^2(\Omega_{J(0)})}^2 \end{equation}
where $ J(0) $ is the value of the multiplicity function defined in \eqref{multfunc} at point $z=0$. 

From Theorem \ref{mainthm}, for any $\epsilon>0$ there exists   a finite linear combination $\wi h \in U\HH$ of vectors $\{\wi g\}_{g\in \G}$ such that  $\mu_{J(0)}(\{0\})\| \wi f_{J(0)}(0)-\wi h_{J(0)}(0)\|_{\ell^2(\Omega_{J(0)})}<\frac \epsilon 2$. Define  $\wi w:=\wi f-\wi h$. Using \eqref{delta-aprox} for $w $, we can pick $\delta$ so small that $\|\wi w-\wi w\mathbbm{1}_{\overline {\mathbb D}_{\delta}^c}\|^2_{L^2(\mu)}< \frac \epsilon 2$.  Let $\wi u$ be a finite linear combination of $ \{z^n\wi g\}_{g\in \G,n\geq m} $ such that $\|\wi w\mathbbm{1}_{\overline {\mathbb D}_{\delta}^c}- \wi u\|^2_{L^2(\mu)}< \frac \epsilon 2$. Then  $\|\wi w-\wi u\|^2_{L^2(\mu)}<  \epsilon $, i.e.,  $\|\wi f -\wi h-\wi u\|^2_{L^2(\mu)}<  \epsilon $.  Hence in this case we get  that any vector $f \in \HH$ is in the closure of the span of $\{\wi g\}_{g\in \G}\cup \{z^n\wi g\}_{g\in \G,n\geq m}= \{z^n \wi g\}_{g\in \G,n=0,m,m+1,\dots} $.	 
\end{proof}	
If we remove the normality condition in Corollary \ref {NotRB}, then for the unilateral shift operator $S$ on $\ell^2(\mathbb{N})$ we have $S^ne_1=e_n$, where $e_n$ is the $n$-th canonical basis vector, i.e.,  in this case the iterated system is not only a Riesz basis, but an orthonormal basis.

Even though we cannot have  bases for $\HH$ by iterations of a countable system $\G$ by a normal  operator, when the system  $\left\{A^ng\right\}_{ g\in \G,\;0\leq n< L(g)}$ is complete, the non-minimality suggest that we may still have a situation in which the system is a frame leading us to the next section.

\section{Complete Bessel systems and frames of iterations} \label {BS}

 It is shown in \cite {ACMT14} that it is possible to construct frames from iteration $\{A^n g\}_{n\ge0}$ of a single vector $g$  for some special cases when the  operator $A$ is an infinite matrix acting on $\ell^2(\N)$, has  point spectrum and $g$ is chosen appropriately \cite{ACMT14}. However, it is also shown that generically, $\{A^n g\}_{n\ge0}$ does not produce a frame  for $\ell^2(\N)$. Since a frame  must be a Bessel system, we study the Bessel properties of $\{A^n g\}_{g \in \G, \;n\geq 0}$ when $A$ is normal. In addition,  we find conditions that must be satisfied when the system $\{A^n g\}_{g \in \G, \;n\geq 0}$ has the lower frame bound property for the case where $\G$ is finite.
 
Denote by $\D_r$ the open disk in $\CC$ of radius $ r $ centered at the origin, by $\overline{\D}_r$ its clousure, and by  $S_r$ its boundary, that is $S_r = \overline{\D}_r \setminus \D_r$. For a set $E\subset \CC$ we will  use the  notation $\CC\setminus E$ or $E^c$ for the complement of $E$. Then we have the following theorem.

\begin{theorem} \label{framecond} Let $A\in \mathcal B(\HH)$ be a normal operator,  $\mu$ be  its scalar spectral measure, and $ \G $ a countable system of vectors in $ \HH $.
 	\begin{enumerate}[(a)]
		\item If  $\{A^n g\}_{g\in \G, \;n\geq 0}$ is complete in $\HH$  and for every $g\in \G$ the system  $\{A^n g\}_{n\geq 0}$ is Bessel in $\HH$, then	$\mu\left(\CC\setminus \overline{\D}_1\right)=0$ and $\mu|_{S_1}$ is absolutely continuous with respect to  arc length measure (Lebesgue measure) on $S_1$.
		\item If $\{A^n g\}_{g\in \G, \;n\geq 0}$ is frame in $\HH$, then	$\mu\left(\CC\setminus \D_1\right)=0$.
%		\item \label{condb}	If  $|\G|<\infty$ and  $\{A^n g\}_{g\in \G, \;n\geq 0}$ satisfies the lower frame bound then, for every $0<\epsilon<1$, $ \textcolor{red}{\mu\left(\overline{\D}_{1-\epsilon}^c\right)>0}$.
	\end{enumerate}
\end{theorem} 
The converse of Theorem \ref{framecond}  is true  in the following sense.
\begin{theorem}\label{invframe} Let $A\in \mathcal B(\HH)$ be a normal operator,  and $\mu$ be  its scalar spectral measure. 
		\begin{enumerate}[(a)]
\item If  $\mu\left(\CC\setminus \overline{\D}_1\right)=0$ and $\mu|_{S_1}$ is absolutely continuous with respect to  arc length measure on $S_1$, then there exists a countable set $ \G\subset \HH $ such that  $\{A^n g\}_{g\in \G, \;n\geq 0}$ is a complete Bessel system.
 \item If $\mu\left(\CC\setminus \D_1\right)=0$ then there exists a countable set $ \G\subset \HH $ such that  $\{A^n g\}_{g\in \G, \;n\geq 0}$ is a Parseval frame for $\HH$.
		\end{enumerate}
    	\end{theorem}

\begin {example} Let $A$ be the convolution operator as in Example \ref{examp1}. If there exists a complete Bessel system by iterations of $A$, then from Theorem \ref{framecond} (a), $\hat a(0)\leq 1$. Conversely, if $\hat a(0)\leq 1$, then the conditions in Theorem \ref{invframe} (b) are satisfied and hence there exists a set of vectors $\G\subset L^2(\R)$ such that $\{A^n g\}_{g\in \G, \;n\geq 0}$ is a Parseval frame in $ L^2(\R) $. From the proof of the theorem, to construct the set $\G$, we take an orthonormal basis $\mathcal{O} $ in $cl\left((1-|\hat a|^2)^\frac{1}{2}L^2(\R)\right)=L^2(\R)$, then $\G=(1-|\hat a|^2)\mathcal{O}$. Note that $\G$ is already complete in $L^2(\R)$. A natural question will be, what is the smallest $\G$ (in terms of its span closure) such that  $\{A^n g\}_{g\in \G, \;n\geq 0}$ is a frame? We will  see from Theorem \ref{cor53}, that $\G$ can not be finite for such a convolution operator since its spectrum is continuous. 
\end{example} 

Using the previous two theorems, we get the following necessary and sufficient conditions for the system  $\{A^n g\}_{g\in \G, \;n\geq 0}$ to be a complete Bessel system in $\HH$.

\begin {corollary}
Let $A\in \mathcal B(\HH)$ be a normal operator,  and $\mu$ be  its scalar spectral measure. Then the following are equivalent.
\begin {enumerate}
\item There exists a countable set $\G \subset \HH$ such that $\{A^n g\}_{g\in \G, \;n\geq 0}$ is a complete Bessel system.
\item $\mu\left(\CC\setminus \D_1\right)=0$  and $\mu|_{S_1}$ is absolutely continuous with respect to  arc length measure on $S_1$. 
\end{enumerate}
\end {corollary}

For the case of iterates  $\{A^n g\}_{g\in \G, \;0\le n < L(g)}$,  where $L \in \mathcal L$  as defined in Remark \ref {ClassL}, one has the following theorem.

\begin {theorem} \label {countable support} Let $A$ be a normal operator on a Hilbert space $\HH$ and $ \G $ a system of vectors in $ \HH $, and assume $L \in \mathcal L$. If $\{A^n g\}_{g\in \G, \;0\le n < L(g)}$ is a complete Bessel system for $\HH$, then for each
$g \in \G$ with $L(g)=\infty$, the set $\{x\in \overline{\D}_1^c| \, \wi g(x)\ne 0 \}$ has $\mu$-measure $0$.

\end {theorem}

When the system $\{A^n g\}_{g \in \G, \;n\geq 0}$ has the lower frame bound property  and  $\G$ is finite, we have the following necessary condition.

\begin {theorem}
\label{condb}	Let $A\in \mathcal B(\HH)$ be a normal operator,  and $\mu$ be  its scalar spectral measure.  If  $|\G|<\infty$ and  $\{A^n g\}_{g\in \G, \;n\geq 0}$ satisfies the lower frame bound, then, for every $0<\epsilon<1$, $ \mu\left(\overline{\D}_{1-\epsilon}^c\right)>0$.
\end{theorem}		
	As a corollary of \ref{framecond}, we get that
\begin{theorem}\label{cor53}
		Let $A$ be a bounded normal operator in an infinite dimensional Hilbert space $\HH$.  If  the system of vectors  $\{A^n g\}_{g\in \G, \;n\geq 0}$  is a frame for some $ \G \subset \HH$ with $|\G| < \infty$, then  $A=\sum_j\lambda_jP_j$ where $ P_j $ are projections such that $\rank P_j\leq |\G|$ (i.e. the global multiplicity of $A$ is less than or equal to $ |\G| $). 	\end{theorem}
		
Combining Theorem \ref {cor53} with the result in \cite {ACMT14}, where $A$ was assumed to be a diagonal operator on $\ell^2(\N)$, we get the following characterization for a general normal operator  $A\in \mathcal B(\HH)$, when $|\G|=1$.  
\begin{theorem} \label {OnePointFrame}
		Let $A$ be a bounded normal operator in an infinite dimensional Hilbert space $\HH$. Then  $\{A^n g\}_{n\geq 0}$ is a frame for $\HH$ if and only if
		\begin{enumerate}
		 \item  [i)] $A=\sum_j\lambda_jP_j$, where $ P_j $ are rank one orthogonal projections.
			\item[ii)] $|\lambda_k| < 1$  for all $k.$  
			\item[iii)]  $|\lambda_k| \to 1$.
			\item[iv)] $\{\lambda_k\}$ satisfies Carleson's condition 
			\begin{equation}
			\label{carleson-cond}
			\inf_{n} \prod_{k\neq n} \frac{|\lambda_n-\lambda_k|}{|1-\bar{\lambda}_n\lambda_k|}\geq \delta.
			\end{equation}
			for some $\delta>0$.
			\item[v)]  $0<C_1\le \frac {\|P_jg\|} {\sqrt{1-|\lambda_k|^2}}  \le C_2< \infty$, for some constants $C_1,  C_2$.
		\end{enumerate}
	\end{theorem}

\begin {example} 
   Let $\HH=\ell^2(\N)$, $A$ a semi-infinite diagonal matrix whose entries are given by $a_{jj}=\lambda_j=1-2^{-j}$ for $j \in \N$, and let $g\in \ell^2(\N)$ be given by $g(j)=\sqrt {1-\lambda_j^2}$. Then, the sequence  $\lambda_j=1-2^{-j}$ satisfies Carleson's condition (see e.g. \cite {Hay58}), and $g$ satisfies condition (v). Thus,  $\{A^n g\}_{n\geq 0}$ is a frame for $\ell^2(\N)$.
\end{example}
\begin {remark}
The following  problem is still open in full generality: Let  $A=\sum_j\lambda_jP_j$ with $|\lambda_j| < 1$  for all $j$, $|\lambda_j| \to 1$, and   $\sup_j\rank P_j<\infty$.   Does there exist a set $\G$ with $|\G|<\infty$ such that $\{A^n g\}_{g\in \G,n\geq 0}$ is a frame for $\HH$? 
\end{remark}
For the special case defined by \eqref {lg}, we get the following necessary condition on the measure $\mu$.

\begin{theorem} \label {liTHM}
Suppose  $A$ is a normal operator, and  $\{A^n g\}_{g \in \G, \;n=0,1,\dots,l(g)}$ (where $l(g)$ is given by \eqref {lg})  is a complete Bessel system for $\HH$. Then \\
(a) If $l(g)=\infty$ then $\{x\in \overline {\mathbb {D}}_1^c: \, \wi g(x)\ne 0 \}$ has $\mu$-measure $0$. \\ %$\text{\supp}(g_i)\subseteq (-1,1)$.\\
(b) The restriction of $\mu$ on $\overline {\mathbb {D}}_1^c$ is concentrated on at most a countable set, i.e., either  $\mu(\overline {\mathbb {D}}_1^c)=0$, or there exists a countable set $E\subset \overline {\mathbb {D}}_1^c$ such that $\mu|_{\overline {\mathbb {D}}_1^c}\left( E^c\cap \overline {\mathbb {D}}_1^c\right)=0$.\\ 
(c) $\mu|_{S_1}$ is a sum of a discrete and an absolutely continuous measure (with respect to arc length measure) on $S_1$.
	\end{theorem}
	
	In fact, if for every $g\in \G$, $l(g)<\infty$  then without the condition that the system is Bessel, but with the completeness condition alone, we get that the measure $ \mu $ is concentrated on a countable subset of $ \CC $, as stated in the following theorem. 
\begin{theorem}
Let $A$ be a normal operator and $ \G\subset\HH $ be a system of vectors such that, for every $g\in \G$, $l(g)<\infty$ and  $\{A^n g\}_{g\in \G, \;n=0,1,\dots,l(g)}$ is complete in $\HH$. Then there exists a countable set $E\subset \CC$ such that $\mu\left(E^c\right)=0$. Moreover, every $g$ is supported, with respect to the measure $\mu$, on a finite set of cardinality not exceeding $l(g)$.
\end{theorem}

\subsection {Proofs of Theorems in Section \ref {BS}}
\begin{proof} [Proof of Theorem \ref {framecond}]
$ (a) $ Suppose 	$\mu(\overline {\mathbb D}_1^c)>0$, then $\mu_k(\overline {\mathbb D}_1^c)>0$ for some $k$, $0\leq k\leq \infty$. Thus, there exists $\epsilon>0$ such that  $\mu_k(\overline {\mathbb D}_{1+\epsilon}^c)>0$. 
Since the system of vectors $\{A^n g\}_{g\in \G, \;n\ge 0}$ is complete in $\HH$, it follows  from Theorem \ref{mainthm} that there exists a $g\in \G$,  such that 
$\mu_k(\overline {\mathbb D}_{1+\epsilon}^c\cap \supp(\wi g_k))>0$.

Let $f\in \HH$ be any vector such that $\wi f=P_k \wi f $, and $\wi f(z)=0$ for $z\in \overline {\mathbb D}_{1+\epsilon}$. Then	
\begin{eqnarray*}
\lefteqn{|\la f,A^ng\ra |} \\
& = &\left|\sum_{0\leq j\leq \infty}\int_\CC z^n \la\wi g_j(z),\wi f_j(z)\ra_{\ell^2(\Om_j)} d\mu_j(z)\right| \\
& = &\left|\int_{\overline {\mathbb D}_{1+\epsilon}^c\cap \;\supp(\wi g_k)} z^n \la \wi g_k(z),\wi f_k(z)\ra_{\ell^2(\Om_k)} d\mu_k(z)\right| .
\end{eqnarray*}
 For each $n$, denote by $\lambda_n(f)$ the  linear functional on the space $\HH_0:=\{f\in \HH: \wi f= P_k\wi f, \; \wi f(z)=0  \text{ for } z\in \overline {\mathbb D}_{1+\epsilon}\}$, defined by  $\lambda_n(f)=\la f, A^ng\ra $. The  norm of this functional (on $\HH_0$) is
 \begin{align*} \|\lambda_n\|^2_{op} &= \int_{\overline {\mathbb D}_{1+\epsilon}^c\cap\; \supp(\wi g_k)} |z|^{2n} \|\wi g_k(z)\|^2_{\ell^2\{\Omega_k\}} d\mu_k(z)\\
 & \geq( 1+\epsilon)^{2n} \int_{\overline {\mathbb D}_{1+\epsilon}^c\cap \;\supp(\wi g_k)}  \|\wi g_k(z)\|^2_{\ell^2\{\Omega_k\}} d\mu_k(z).
 \end{align*}
  Since the right side of the last inequality tends to infinity as  $n\rightarrow\infty$, so does  $\|\lambda_n\|_{op}$. Thus, from the uniform boundedness principle there exists an $f\in \HH_0$ such that 
 $$\lim_{n\rightarrow\infty}\left|\int_{\overline {\mathbb D}_{1+\epsilon}^c\cap \;\supp(\wi g_k)} z^n \la \wi g_k(z),\wi f_k(z)\ra _{\ell^2(\Om_k)} d\mu_k(z)\right| =  \infty.$$
For such  $f$, $\lambda_n(f)=|\la f,A^ng\ra \|\rightarrow\infty$ as $n\rightarrow\infty$. Thus, we also have that
 $$ \sum_{n=0}^{\infty}|\la f,A^ng \ra|^2=\infty$$
 which is a contradiction to our  assumption that $\{A^n g\}_{n\ge 0}$ is a Bessel system in $\HH$. 
 
 To prove the second part of the statement, let $k \geq 1$ be fixed, and consider the Lebesgue decomposition of $\mu_k|_{S_1}$ given by $\mu_k|_{S_1}=\mu_k^{\rm  ac}+\mu_k^{s}$ where $ \mu_k^{\rm ac} $ is absolutely continuous with respect to arc length measure on $S_1$, $\mu_k^{s}$ is singular and $\mu_k^ {\rm ac}\perp \mu_k^{s}$. We want to show that $\mu_k^{s}\equiv 0$.  
 
    For a vector $a=(a_1,a_2,\dots,a_k)\in \ell^2(\Om_k)$ define $Q_ra:=a_r$.   Fix  $1\leq r\leq k$ and $m\geq 1$. 
   Let $ f\in \HH $ be the vector such that 
   \begin{enumerate}[i)]
\item $Q_r\wi f_k(e^{2\pi it})=e^{2\pi imt}$, $ \mu_k^s $-a.e.
\item $Q_r\wi f_k(e^{2\pi it})=0$, $ \mu_k^ {\rm ac} $-a.e.
\item $Q_{s}\wi f_j(z)=0$ if $r\neq s$ or $ k\neq j $ 
\item $\wi f(z)=0$ for $z\notin S_1$.
   \end{enumerate}
    Then for such an $ f $ and a fixed $g\in \G$,  from the assumption that  $\{A^n g\}_{n\ge 0}$ is a Bessel system in $\HH$, we have
    \begin{eqnarray*}
\lefteqn{\sum_{n\geq 0}\left|\int_{S_1} e^{2\pi int} Q_r \wi g_k(e^{2\pi it}) \overline{ e^{2\pi imt}} d\mu_k^{s}(e^{2\pi it})\right|^2 }\\
      &=&\sum_{n\geq 0}\left|\sum_{0\leq j\leq \infty}\int_\CC z^n \la \wi g_j(z),\wi f_j(z)\ra_{\ell^2(\Om_j)} d\mu_j(z)\right|^2 \\
     & = & \sum_{n=0}^{\infty}|\la A^ng,f\ra |^2\leq C \|f\|^2\leq C \mu(S_1).
     \end{eqnarray*}
            Thus  
    $$\sum_{n\geq 0}\left|\int_{S_1} e^{2\pi i(n-m)t} Q_r\wi g_k(e^{2\pi it}) d\mu_k^{s}(e^{2\pi it})\right|^2\leq C \mu(S_1).$$
   Since the last inequality holds for every $m\geq 1$,  we have
   $$\sum_{n\in \Z}\left|\int_{S_1} e^{2\pi int} Q_r\wi g_k(e^{2\pi it}) d\mu_k^{s}(e^{2\pi it})\right|^2\leq C \mu(S_1).$$
This means the Fourier-Stieltjes coefficients of the measure $Q_r\wi g_k(e^{2\pi it}) d\mu_k^{s}(e^{2\pi it})$ are in $\ell^2(\Z)$.  Hence, from the uniqueness theorem of the Fourier Stieltjes coefficients (\cite{kaznelson}, p. 36) and the fact that any element of $\ell^2(\Z)$ determines Fourier coefficients of an $L^2(S_1)$  function (with respect to arc length measure), $Q_r\wi g_k(e^{2\pi it}) d\mu_k^{s}(e^{2\pi it})$ is absolutely continuous with respect to the arc length measure. But the measure $\mu_k^{s}$ is concentrated on a measure zero set as a singular measure, hence
$Q_r \wi g_k(e^{2\pi it}) d\mu_k^{s}(e^{2\pi it})$ is  the zero measure. Since the system  $\{A^n g\}_{g\in \G, \;n\geq 0}$ is complete in $\HH$,  from Theorem \ref{mainthm} we obtain that $\mu_k^s=0$ and hence $\mu_k$ is absolutely continuous with respect to the  arc length measure on $S_1$. Thus $\mu$ is absolutely  continuous with respect to the  arc length measure on $S_1$.

$ (b) $ Suppose $\{A^n g\}_{g\in \G, \;n\geq 0}$ is a frame with frame bounds $\alpha$ and $ \beta $.    Let $ f\in \HH $ be any vector such that $\wi f=0$ on $\CC\setminus S_1$. For such an $f$ we have that $\|(A^*)^mf\|=\|A^mf\|=\|f\|$ for any $m \in \Z $. Thus, for any $m \in \Z $, we  have
	\begin {eqnarray}
	\label {UP}
\;\alpha\|f\|=\alpha\|(A^*)^mf\| &\leq&	\sum_{g\in \G}\sum_{n= 0}^{\infty}|\la(A^*)^mf,A^ng\ra|^2
=\sum_{g\in \G}\sum_{n= 0}^{\infty}|\la f,A^{n+m}g\ra|^2\\
	&=&\sum_{g\in \G}\sum_{n= m}^{\infty}|\la f,A^{n}g\ra|^2\leq \beta \|A^mf\|\leq \beta\|f\|. \nonumber
	\end {eqnarray}
    	Since \eqref {UP} holds for every $m$, the right inequality implies $\sum_{n= m}^{\infty}\sum_{g\in \G}|\la f,A^{n}g\ra|^2\to 0$ as $m\to \infty$.  Hence, using the left inequality we conclude that $\|f\|=0$.  Since $f$ is such that $\wi f=0$ on  $ \CC\setminus S_1 $, but otherwise is arbitrary, it follows that $\mu(S_1)=0$. But, from Part (a), we already know that $ \mu(\CC\setminus \overline{\mathbb D}_1) =0$, hence $\mu(\CC\setminus \mathbb D_1) =0$.
	\end{proof}
	
	\begin{proof}[Proof of Theorem \ref{invframe}] 
	%\textcolor{red}{(a)    Let $\HH=\oplus_{i\in I}\HH_i$ where $\HH_i$ are subspaces of $\HH$ indexed by $i\in I$ where $I$ is a countable indexing set. Let $\G_i\subset \HH_i$, be a complete Bessel system in $\HH_i$ for each  $i\in I$ with uniform frame bound $B$.   Then, it is not difficult to see that  $\cup_{i\in I}\G_i$  is a complete Bessel system in $\HH$ with the Bessel bound $B$. } 
		
Let $\HH_1=\{f\in\HH: \wi f(z)=0, z\notin S_1\}$ and $\HH_2=\{f\in\HH: \wi f(z)=0, z\notin \mathbb D_1\}$. Then $\HH=\HH_1\oplus \HH_2$. Let $\G_i\subset \HH_i$, be complete Bessel systems in $\HH_i$, $i=1,2$, then, it is not difficult to see that  $\G_1\cup\G_2$  is a complete Bessel system in $\HH$.     We will proceed by constructing complete Bessel systems for $\HH_1$ and $\HH_2$.
To construct a complete Bessel sequence  for $\HH_1$, we first consider the operator $N_{\mu_j|_{S_1}}$ on $L^2(\mu_j|_{S_1})$ for a fixed $j$, with $1\le j\le \infty$, where $\mu_j$ is as in the decomposition of Theorem \ref{decomp}. Since for $f \in \HH_1$, $\wi f(z)=0$ for  $z\notin S_1$, and since $\mu|_{S_1}$ (and hence also $\mu_j|_{S_1}$) is absolutely continuous with respect to  the arc lengh measure $\sigma$, we have that on the circle $S_1$, $d\mu_j|_{S_1}=w_jd\sigma$ for some $w_j \in L^1(\sigma)$. Hence on the support $E_j$ of $w_j$,  $\mu_j$ and $\sigma$ are mutually absolutely continuous, i.e., for $\nu_j$ defined by $d\nu_j=\mathbbm{1}_{E_j}d\sigma$, $\mu_j$  and $\nu_j$ are mutually absolutely continuous. 

Now consider the two functions $p_j$ and $q_j$ such that $p_j(z)=q_j(z)=0$ for  $z\notin S_1$, while on $S_1$, $p_j(e^{2\pi i t})=\mathbbm{1}_{\left[0,\frac{1}{2}\right]}(t)$ and $q_j(e^{2\pi i t})=\mathbbm{1}_{\left[\frac{1}{2},1\right]}(t)$. From the properties of the Fourier series on $L^2(S_1,\sigma)$, the sets  $\{z^np_j(z)\}_{n\geq 0}$ and $\{z^nq_j(z)\}_{n\geq 0}$ are Bessel systems in $L^2(S_1,\sigma)$ with bound $1$.  Thus, $\{z^np_j(z)\}_{n\geq 0}$ and $\{z^nq_j(z)\}_{n\geq 0}$ are also Bessel systems in $L^2(S_1,\nu_j)$  with bound $1$. Therefore, $\{z^np_j(z)\}_{n\geq 0}\cup \{z^nq_j(z)\}_{n\geq 0}$ is a Bessel system for $\oplus_{j=1}^\infty L^2(S_1,\nu_j)$.  By Proposition \ref {emptintspec} and Theorem \ref{decomp} the sytem $\{z^np_j(z)\}_{n\geq 0}\cup \{z^nq_j(z)\}_{n\geq 0}$ is also complete in $\oplus_{j=1}^\infty L^2(S_1,\nu_j)$. 
Thus, $\{z^np_j(z)\}_{n\geq 0}\cup \{z^nq_j(z)\}_{n\geq 0}$ is a complete Bessel system for $\oplus_{j=1}^\infty L^2(S_1,\nu_j)$. 

Since $\mu_j$ and $\nu_j$ are mutually absolutely continuous,  the multiplication operator $z$ on $\oplus_{j=1}^\infty L^2(S_1,\nu_j)$ is unitarily equivalent to the multiplication operator $z$ on $\oplus_{j=1}^\infty L^2(S_1,\mu_j)$ which we denote by $V$. Hence,  $\{z^nV(p_j)(z)\}_{n\geq 0}\cup \{z^nV(q_j)(z)\}_{n\geq 0}$ is a complete Bessel system for $\oplus_{j=1}^\infty L^2(S_1,\mu_j)$.  Finally, using  Theorem \ref{decomp}, it follows  that $\{A^nU^{-1}V(p_j(z))\}_{n\geq 0}\cup \{A^nU^{-1}V(q_j(z))\}_{n\geq 0}$ forms a complete Bessel system for $U\HH_1=\oplus_{j=1}^\infty L^2(S_1,\mu_j)$. 

The existence of complete Bessel system in $ \HH_2 $ (moreover, a Parseval frame) follows from Part $ (b) $ of Theorem \ref{invframe} which we prove next.

$ (b) $  Let $D$ be the operator $(I-AA^*)^{-\frac 1 2 }$. Let $\mathcal O$ be an orthonormal basis for $ cl(D\HH) $, and define $\G=\{g=Dh: h \in \mathcal O\}$. Then 
%$\mathcal{V}=cl \{\sqrt{1-|z|^2}\cdot \wi h(z):h\in\HH\}$. Let $\{h\}_{h\in \mathcal I}$ be an orthonormal basis for $ \mathcal{V} $. Then
		\begin {eqnarray*} 
	\sum_{n=0}^{m}\sum_{h\in \mathcal O}|\la f,A^nDh\ra |^2&=&\sum_{n=0}^{m}\sum_{h\in \mathcal O}|\la D(A^*)^nf,h\ra |^2
	=\sum_{n=0}^{m}\|D(A^*)^nf\|^2\\
	&=&\sum_{n=0}^{m}\la D^2(A^*)^nf,(A^*)^nf\ra 
	=\sum_{n=0}^{m}\la (I-AA^*)(A^*)^nf,(A^*)^nf\ra \\
	&=&\|f\|^2-\|(A^*)^{m+1}f\|.
	\end {eqnarray*}
	Using Lebesgue's Dominated Convergence Theorem,  $\|(A^*)^mf\|^2=\int_{\overline{\mathbb D}_1}|z|^{2m}\|\wi f(z)\|^2d\mu(z)\to 0$ as $m\to \infty$ since $|z|^{2m}\to 0,$ $\mu-a.e.$ on ${\overline{\mathbb D}_1}$.  Hence, from the  identity above we get that 	
	$$\sum_{n=0}^{\infty}\sum_{h\in \mathcal{I}}|\la f,A^nDh\ra |^2=\|f\|^2.$$
	Therefore  the system of vectors $\G=\{g=Dh: h \in \mathcal O\}$ is a tight frame for $\HH$.
\end{proof}
	
The proof of Theorem \ref{countable support} is a direct consequence of the proof of (a) in the above theorem.

\begin {proof} [Proof of Theorem \ref {condb}]
Suppose 	$\mu(\overline {\mathbb {D}}_{1-\epsilon}^c)=0$ for some $0<\epsilon<1$. Because $|\G|<\infty$ and $\dim(\HH)=\infty$, the system $\{A^n g\}_{g\in \G, \;n=0,1,\dots, M}$ is not complete in $\HH$ for $M<\infty$. From the Hahn-Banach theorem, there exists a vector $h\in \HH$ with $\norm{h}=1$ such that $\la A^n g,h\ra=0$ for every $g \in \G,$ and $n=0,\dots,M$. Then 
\begin{eqnarray*}
\lefteqn{\sum_{g \in \G} \sum_{n=0}^{\infty}|\la h,A^ng\ra |}\\
&=&
	\sum_{g \in \G} \sum_{n=M+1}^{\infty}\Big|\sum_{0\leq j\leq \infty}\int_\CC z^n \la \wi h_j (z),\wi g_j(z)\ra _{\ell^2(\Om_j)} d\mu_j(z)\Big |^2\\
	&\leq&\sum_{g \in \G} \sum_{n=M+1}^{\infty}\Big |\int_{\overline {\D}_{1-\epsilon}} z^n \sum_{0\leq j\leq \infty} \mathbbm{1}_{\mathcal{E}_j}(z) \la \wi h_j(z),\wi g_j(z)\ra _{\ell^2(\Om_j)} d\mu(z)\Big|^2\\
&\leq& \sum_{g \in \G} \sum_{n=M+1}^{\infty}(1-\epsilon)^{2n} \Big (\sum_{0\leq j\leq \infty}\int_\CC  |\la \wi h_j (z),\wi g_j (z)\ra _{\ell^2(\Om_j)}| d\mu_j(z)\Big)^2.
\end{eqnarray*}
	Applying H\"older's inequality several times, we get
\begin{eqnarray*}
\lefteqn{\sum_{0\leq j\leq \infty} \int_\CC |\la \wi h_j(z),\wi g_j(z)\ra_{\ell^2(\Om_j)}| d\mu_j(z)}\\
&\leq&\sum_{0\leq j\leq \infty} \int_\CC \norm{\wi h_j(z)}_{\ell^2(\Om_j)} \norm{\wi g_j(z)}_{\ell^2(\Om_j)}  d\mu_j(z)\\
&\leq&\sum_{0\leq j\leq \infty} \Big (\int_\CC\norm{\wi h_j(z)}^2_{\ell^2(\Om_j)}d\mu_j(z)\Big)^{\frac{1}{2}}\Big (\int_\CC \norm{\wi g_j(z)}^2_{\ell^2(\Om_j)}   d\mu_j(z)\Big)^{\frac{1}{2}}\\
&\leq &\norm{h}\norm{g}.
\end{eqnarray*}
%%%%%%%%%%%%%%%%%%%%%%%%%%%%%%%%%%
Hence,
\begin{align*}
%&&\lefteqn{\sum_{g \in \G} \sum_{n=0}^{\infty}|<h,A^ng>|^2} \leq 
&\sum_{g \in \G} \sum_{n=0}^{\infty}|\la h,A^ng\ra|^2 \leq
\sum_{g \in \G} \sum_{n=M+1}^{\infty}(1-\epsilon)^{2n} \norm{h}^2\norm{g}^2 \\
=&\frac{(1-\epsilon)^{2(M+1)}}{1-(1-\epsilon)^{2}}\norm{h}^2\sum_{g \in \G} \norm{g}^2\;\rightarrow 0\;\;\; \text{as}\; \;M \rightarrow \infty.
\end{align*}
	Therefore the left frame inequality does not hold, and we have a contradiction.
\end{proof}

%The proof of Theorem \ref{countable support} is a direct consequence of the proof of (a) in the above theorem.
\begin{proof}[Proof of Theorem \ref{cor53}]
Define the subspace $ V_\rho$ of $ \HH$ to be $ V_\rho=\{ f: \supp \tilde f \subseteq \overline {\mathbb D}_\rho\}$. The restriction  of $A$  to $ V_\rho$ is  a normal operator with its spectrum equal to the part of the spectrum of $A$ inside $ \overline {\mathbb D}_\rho$.  Let $\widetilde \G=U\G$ where $U$ is as in Theorem \ref {decomp}.  Let $\widetilde {\G}_\rho=\{\mathbbm{1}_{\overline {\mathbb D}_\rho}\wi g: \wi g \in \widetilde \G\}$. Since $\{A^ng\}_{g\in \G, n\ge0}$ is a frame by assumption, $\{z^n \wi w\}_{\wi w \in \widetilde {\G}_\rho}
$ is a frame for $ UV_\rho$.  Thus, since $\rho<1$, Theorem \ref {condb} implies that $V_\rho$ is finite dimensional. Hence the restriction of the spectrum of $A$ to $\overline {\mathbb D}_\rho$ for any $\rho<1$ is a finite set of points. We also know from Theorem \ref {framecond}  (b) that $\mu(\mathbb {D}_1^c)=0$.   Thus,  $UAU^{-1}$ has the form $\Lambda=\sum_j\lambda_jP_j$. 
\end{proof}
%\begin{proof}[Proof of Theorem \ref{invframe}] \textcolor{red}{(a)
%Let $\HH_1=\{f\in\HH: \wi f(z)=0, z\notin S_1\}$ and $\HH_2=\{f\in\HH: \wi f(z)=0, z\notin D^o_1\}$. Then $\HH=\HH_1\oplus \HH_2$. Notice that if we have a complete Bessel system in $\HH$ and a complete Bessel system in $\HH_2$, then the union of these two systems is a Bessel system in $\HH_1\oplus\HH_2$. }

%\textcolor{red}{The existance of complete Bessel system in $ \HH_2 $ (moreover, a Parseval frame) follows from $ (b) $.}

%\textcolor{red}{$ (b) $  $\mathcal{V}=cl \{\sqrt{1-|z|^2}\cdot \wi h(z):h\in\HH\}$. Let $\{h\}_{h\in \mathcal I}$ be an orthonormal basis for $ \mathcal{V} $. Then
%		\begin {eqnarray*} \sum_{n=0}^{m}\sum_{h\in \mathcal I}|<f,A^nDh>|^2&=&\sum_{n=0}^{m}\sum_{h\in \mathcal I}|<D(A^*)^nf,h>|^2\\
%	&=&\sum_{n=0}^{m}\|D(A^*)^nf\|^2\\
%	&=&\sum_{n=0}^{m}<D^2(A^*)^nf,(A^*)^nf>\\
%	&=&\sum_{n=0}^{m}<(I-AA^*)(A^*)^nf,(A^*)^nf>\\
%	&=&\|f\|^2-\|(A^*)^{m+1}f\|.
%	\end {eqnarray*}
%	Taking limits as $m\to \infty$ and using the fact that $(A^*)^mf\to 0$ we get from the  identity above that 	
%	$$\sum_{n=0}^{\infty}\sum_{h\in \mathcal{I}}|<f,A^nDh>|^2=\|f\|^2.$$
%	Therefore  the system of vectors $\G=\{g=Dh: h \in \mathcal I\}$ is a tight frame for $\HH$.}
%\end{proof}

\begin{proof}[Proof of Theorem \ref {liTHM}]

(a) Follows from Theorem \ref {countable support}. 

(b) If $l(g) < \infty$ then $A^{l(g)}g - \sum_{k=0}^{l{(g)}-1} c_k A^kg = 0$ for some complex numbers $c_k$.
Call $Q$ the polynomial $Q(z) := z^{l(g)} - \sum_{k=0}^{l(g)-1} c_k z^k.$ We have $Q(A)g=0$ and therefore $0 =U(Q(A)g)(z) = Q(z) \wi g(z)$ $\mu-$a.e. $z$. 

Let $E_g$ be the set of roots of $Q$.
Hence $\wi g(z) =0$ $\mu$ a.e. in $(\CC \setminus E_g).$
This together with  part (a) of the theorem gives us that, for all  $g \in \G$, 
\begin{equation} \label{zero-set}
  \wi g(z) = 0 \;\;\text{ a.e. } \mu \;\;\text{ in } \bigcap_{g\in \G_F}(\overline {\mathbb {D}}_1^c \setminus E_g)
\end{equation}
where, $\G_F=\{g\in \G: l(g) < \infty\}.$

The set  $E:=\bigcup_{g\in \G_F} E_g$ is countable and $\bigcap_{g\in \G_F}(\overline {\mathbb {D}}_1^c \setminus E_g) = \overline {\mathbb {D}}_1^c \setminus E.$ So \eqref{zero-set} holds on $\overline {\mathbb {D}}_1^c \setminus E.$
It follows that for each $j \in \N^*$,
$\spn\{\wi g_j(z)\}_{g\in \G}$ is not complete in $\ell^2(\Omega_j)$ $\mu_j-{\rm a.e.} \;z \in  \overline {\mathbb {D}}_1^c \setminus E$ and therefore $\mu_j (\overline {\mathbb {D}}_1^c \setminus E) =0$.
We conclude that $\mu(\overline {\mathbb {D}}_1^c \setminus E)=0$.

(c) Let $\Delta:=\{x \in S_1: x \in \supp g,\, g \in \G, \, l(g)<\infty\}$.   From the proof of (b) $\Delta$  is countable. Then, since the projection of a Bessel system is Bessel, and the projection of a complete set is complete,   following the proof of Theorem \ref{framecond}(a) we can see that $\mu$ is absolutely continuous on  $S_1\setminus \Delta$.   
\end{proof}

\section{Self-adjoint operators} \label {SA}

The class of self-adjoint operators is an important subclass of normal reductive operators which has some interesting properties that we study in this section.  In particular,  
we prove that for self-adjoint operators the normalized system  $\left\{\frac{A^n g}{\norm{A^n g}}\right\}_{g \in \G,\;n\geq 0}$ is never a frame. The proof of this fact relies on the following theorem.

\begin {theorem}\label{feichtinger}  Every unit norm frame is a finite union of Riesz basis sequences. 
\end {theorem}

Theorem \ref{feichtinger} was conjectured by Feichtinger and is equivalent  to the Kadison-Singer theorem \cite {CT06, BS06} which was proved  recently in \cite{feichtconj}.

\begin{theorem} \label{normframe}
	If $A$ is a self-adjoint operator on $\HH$ then the system  $\left\{\frac{A^n g}{\norm{A^n g}}\right\}_{g \in \G,\;n\geq 0}$ is not a frame for $\HH$.
\end{theorem}

\begin{remark}
An open problem is whether  the theorem remains  true for general normal operators. The theorem does not hold if the operator is not normal.  For example, the shift operator $S$ on $\ell^2(\N)$ defined by $S(x_1,x_2,\dots)=(0,x_1,x_2,\dots)$, is not normal, and $\{S^ne_1\}$ where $e_1=(1,0,\dots )$ is an orthonormal basis for $\ell^2(\N)$.
	\end{remark}
\begin {remark}\rm
It may be that the system  $\left\{\frac{A^n g}{\norm{A^n g}}\right\}_{g \in \G,\;n\geq 0}$ is not a frame for $\HH$ because it is overly redundant due to the fact that we are iterating $\{A^n g\}_{g\in \G}$ for all $n\ge 0$. We may reduce the redundancy by letting $0\le n <L(g)$ where $L \in \mathcal L$ as defined in Remark \ref {ClassL}.  For example if $\{g\}_{g \in \G}$ is an orthonormal basis for $\HH$, then trivially, we can choose $L(g)=1$ and the system  $\left\{\frac{A^n g}{\norm{A^n g}}\right\}_{g \in \G,0\le n < L(g)}$ is  an orthonormal basis for $\HH$. However, if $\G$ is finite, $\left\{\frac{A^n g}{\norm{A^n g}}\right\}_{g \in \G,0\le n < L(g)}$ cannot be a frame for $\HH$ as in the corollary below. 
\end{remark}
\begin {corollary}
\label {FinVecFr}
Let  $\{g\}_{g \in \G }\subset \HH$ and assume that  $|\G|< \infty$ and $L\in \mathcal L$. Then for a self-adjoint operator $A$, $\left\{\frac{A^n g}{\norm{A^n g}}\right\}_{g \in \G,0\le n < L(g)}$ is not a frame for $\HH$.
\end {corollary}

\begin{proof}[Proof of Theorem \ref {normframe}]  Suppose it is a frame.  	
	Using Feichtinger's theorem,  we decompose the set $\left\{\frac{A^n g}{\norm{A^n g}}\right\}_{g \in \G,\;n\geq 0}$  into a finite union  of Riesz sequences.  Choose a  vector $h\in \G$. Thus the subsystem  $\left\{\frac{A^n h}{\norm{A^n h}}\right\}_{n\geq 0}$ can be decomposed into a union of Riesz sequences and therefore a union of minimal sets.  Since there are finitely many sequences, the powers of $A$ in one of these sequences must contain infinite number of even  numbers $\{2n_k\}$ (in particular, the  system  $\{A^{2n_k} h\}_{k=1,\dots}$ is a minimal set) such that
	\begin{equation}\label{muntzcond}
\sum\limits_{k\ge 1}\frac{1}{n_k}=\infty.
	\end{equation}

	If we consider the operator $A^2$, then its spectrum is a subset of $[0,\infty)$. In order to finish the proof of the Theorem, we use the following Lemma whose proof is a corollary  of the 
	M\"untz-Sz\'asz theorem \cite{rudinrc}. 
	
	\begin{lemma} \label{h=00} Let $\mu$ be a regular Borel measure on $[0,\infty)$ with a compact support and $n_k$, $k=0,1,\dots$, be a sequence of natural numbers such that $n_0=0$ and $$\sum\limits_{k\ge 1}\frac{1}{n_k}=\infty.$$ For a function  $\phi \in L^1({\mu})$, if   		$$\int_0^{\infty} x^{n_k} \phi(x)d\mu(x)=0\; \text{for  every }\;k,$$
		 then  $\phi=0$ $\mu$ a.e..
	\end{lemma}

Let $V=cl({span}\{(A^2)^{n}h\}_{n\ge 0})$, and let $B$ be the restriction of $A^2$ on $V$.  Since $B$ is positive definite, its spectrum $\sigma(B) \subset [0, b]$ for some $b\ge 0$. Let $\mu$ be the measure defined in \eqref{themeas} associated with $B$. By Theorem \ref {mainthm}, $\mu_j=0$ for all $j\ne 1$ (i.e., $\mu=\mu_1$), and $\wi h(x) \ne 0$ a.e. $ \mu$.

Let $ n_k,  k\geq 1 $ be the sequence of integers chosen above such that  $\{A^{2n_k} h\}_{k=1,\dots}$ is a minimal set and \eqref{muntzcond} holds. Set $n_0=0$. Note that  both  sequences $\{n_k\}_{k\ge 0}$, and $\{n_k\}_{k=0, m, m+1, \dots}$ satisfy the condition of the Lemma \ref {h=00}, hence $\int_0^b x^{n_k}\wi h(x)\overline {\wi f (x)} d\mu(x)=0$ for all $k\ge 0$ implies that $\wi f=0$ a.e. $\mu$, as well as $\int_0^b x^{n_k}\wi h(x)\overline {\wi f (x)} d\mu(x)=0$ for all $k=0, m, m+1,\dots$ implies that $\wi f=0$ a.e. $\mu$.  Thus, $V=cl({span}\{(A^2)^{n}h\}_{n\ge 0})=cl({span}\{(A^2)^{n_k}h\}_{k=0,m,m+1,\dots})=cl({span}\{(A^2)^{n_k}h\}_{k\ge0})$ which contradicts the minimality condition. 
\end{proof}

\begin {proof} [Proof of Corollary \ref {FinVecFr}] suppose $\left\{\frac{A^n g}{\norm{A^n g}}\right\}_{g \in \G,\;0\leq n<L(g)}$ is a frame for $\HH$. Because $\dim \HH=\infty$, the set $\G_{\infty}= \{g \in \G |\;L(g)=\infty\}$ is non-empty.  Then the system $\left\{\frac{A^n g}{\norm{A^n g}}\right\}_{g\in \G_{\infty},0\leq n<L(g)=\infty}$ is a frame for its closure since we get it by removing finite number of vectors from a frame. The closure is an invariant subspace and  $A$ restricted to it remains self-adjoint which contradicts Theorem \ref {normframe}.
\end {proof}

%\textcolor{blue}{Does the upper or the lower bound fail that it is not a frame? I guess it is the upper bound as clarified by your remark above}

\section{Applications to groups of unitary operators}\label {GUO}

In this section, we apply some of our results to discrete groups of unitary operators. These often occur in wavelet, time frequency and frame constructions. 

As a corollary of the spectral theorem of normal operators, a normal operator is unitary if and only if its spectrum is a subset of the unit circle. We will need the following Proposition from Wermer  \cite{wermer}.

\begin{proposition}[\cite{wermer}] \label {Wermer1}
For a unitary operator $T$ the following are equivalent
\begin{enumerate}
	\item $ T $ is not reductive

	\item The arc length measure is absolutely continuous with respect to the spectral measure of $T$.
\end{enumerate}
\end{proposition}

Let $\pi$ be a unitary representation of a discrete group $ \Gamma $  on Hilbert space $\HH$. The order $o(\gamma)$ of an element $\gamma\in\Gamma$ is the smallest natural number $m$ such that $\gamma^m=1$. If no such number exists then we say $o(\gamma)=\infty$. The same way we define the order of an operator $\pi(\gamma)$.

Notice that if $o(\pi(\gamma))<\infty$ then 
it is reductive and its spectrum is a subset of the set of $ o(\pi(\gamma)) $-th roots of unity.

\begin{theorem}
Let $\pi$ be a unitary representation of a discrete group $ \Gamma $  on Hilbert space $\HH$ and suppose there exists a set of vectors $\G \subseteq \HH$, such that $\{\pi(\gamma) g:\;\gamma\in \Gamma, g\in \G\}$ is a minimal system. Then, for every $\gamma\in \Gamma$ with $ o(\gamma) =\infty$, $\pi(\gamma)$ is non-reductive and hence the arc length measure on $ S_1 $ is absolutely continuous with respect to the spectral measure of $\pi(\gamma)$.
 \end{theorem}
\begin{proof} The minimality condition implies that $\pi$ is injective and hence $o(\gamma)=o(\pi(\gamma))$.
Let $\gamma\in \Gamma$ be such that $ o(\gamma) =\infty$, then from the minimality assumption, $\{\pi(\gamma)^n g:\;\gamma\in \Gamma, g \in \G\}$ is a minimal subsystem. Thus, from Corollary \ref{reductmin}, $\pi(\gamma)$ is non-reductive. The rest follows from  Proposition \ref  {Wermer1} above.
\end{proof}

\begin{theorem} \label{groupbess}
Let $\pi$ be a unitary representation of a discrete group $ \Gamma $  on Hilbert space $\HH$  and suppose there exists a set of vectors 
$\G \subseteq \HH$, such that $\Gamma\{\G\}=\{\pi(\gamma) g:\;\gamma\in \Gamma, g\in \G\}$ is  complete in $\HH$ and, for every $ g\in \G $, $ \Gamma\{g\}=\{\pi(\gamma) g:\;\gamma\in \Gamma \} $ is a Bessel system in $\HH$. Then for every $ \gamma \in \Gamma$ with $o(\gamma)=\infty$,  the measure $\mu$ associate with $\pi(\gamma)$  is absolutely continuous with respect to the arc length measure on $S_1$.
\end{theorem}
\begin{proof} Suppose $o(\gamma)=\infty$.  The assumption that the system  $ \{\pi(\gamma) g:\;\gamma\in \Gamma \} $ is Bessel implies that the kernel of the representation $ \pi $ must be finite, otherwise any vector in the system will be repeated infinitely many times, prohibiting the Bessel property from holding. Thus  $o(\gamma)=\infty$ implies $o(\pi(\gamma))=\infty$.
	
	 Pick any vector $\pi(h)g$ where $h\in \Gamma, g \in \G$. Then $\{\pi(\gamma)^n \pi(h)g\}_{n\geq 0}$ is a subsystem of $\Gamma\{g\}$ since $\pi(\gamma)^n\pi(h)\neq \pi(\gamma)^m\pi(h)$ if $n\neq m$.  Hence, using the fact that $\{\pi(\gamma)^n \pi(h)g\}_{n\geq 0}$ is a Bessel sequence, from the proof of Theorem \ref {framecond}(a) we  get that, for every $j\in \N^*$, the measure  $\mu_j$ in the \eqref{repres} representation of $\pi(\gamma)$ is absolutely continuous on $\supp [({\pi(h)g})\tilde{}_j]$. Since $\{\pi(h)g:\;h\in \Gamma,g \in \G\}$ is complete in $\HH$, from Theorem \ref{mainthm}, $\mu$ is concentrated on the set $\cup_{0\leq j\leq \infty}\supp [({\pi(h)g})\tilde{}_j]$ thus we  get that  the spectrum of $\pi(\gamma)$ is absolutely continuous with respect to arc length measure. \end{proof}

In fact, it was shown in \cite{Weber08} (lemma 4.19) that the assumptions in the previous theorem hold if and only if $\pi$ is a subrepresentation of the left regular representation of $\Gamma$ with some multiplicity. And as a corollary of that, if the conditions of  Theorem \ref{groupbess} hold, it is possible to find another set $\G^\prime\subset \HH$ such that $\{\pi(\gamma) g:\;\gamma\in \Gamma, g\in \G^\prime\}$ is a Parseval frame for $\HH$.
\section{Concluding remarks}
In this paper we consider a system of iterations of the form $\{A^ng\;|\; g \in \G,\;0\leq n< L(g)\}$ 
where $A$ is a normal bounded operator in a Hilbert space $\HH$, $\G\subset \HH$ is countable set of vectors, and where $L$  is a function defined in Definition \ref {ClassL}. The goal is to find  relations between the operator $A$,  the set $\G\subset \HH$, and the function $L$ that makes the system of iterations $\{A^ng\;|\; g \in \G,\;0\leq n< L(g)\}$ complete, Bessel, a basis, or a frame. Although we have exhibited some of these relations for the case of normal operators and reductive normal operators, there are many open questions. For example, a necessary condition as to when  the system of iterations $\{A^ng\;\;n\ge0\}$ is a frame for $\HH$ is derived, but the necessary and sufficient conditions is only answered for $\HH\in \ell^2(\N)$ and when $A$ is essentially a compact self-adjoint operator (see \cite {ACMT14}). In fact, some of the problems in this work are connected to  the still open invariant subspace problem in Hilbert spaces.  The connection between sampling theory and frame theory, and some of the problems in operator algebras, and spectral theory  makes the Dynamical Sampling problem, which is the underlying problem driving this work, a fertile ground for interaction between these areas of mathematics.

\section{Acknowledgements}
Akram Aldroubi would like to thank the Hausdorff Institute of Mathematics for providing him the resources  to finish this work while in residence at HIM during the special trimester on `` Mathematics of Signal Processing,'' in 2016 Ahmet Faruk   \c{C}akmak was visiting Vanderbilt while collaborating on this research. The authors are very grateful for the thoughtful comments of the referee   which helped them to substantially improve the  paper and its presentation.

\bibliography{refers}{}
\bibliographystyle{plain}

\end{document}